\documentclass[12pt,a4paper,dvipdfmx]{article}
\usepackage{amsmath,amsthm,amssymb,mathrsfs,enumerate,hyperref}
\usepackage[numbers]{natbib}

\def\M{{\mathcal{M}}}
\def\N{{\mathbb{N}}}

\def\R{{\mathbb{R}}}
\def\U{{\mathcal{U}}}

\theoremstyle{plain}
\newtheorem{thm}{Theorem}[section]

\newtheorem{prop}{Proposition}[section]

\theoremstyle{definition}

\newtheorem{rem}{Remark}[section]

\allowdisplaybreaks

\title{Optimality Conditions for Nonconvex Variational Problems with Integral Constraints in Banach Spaces\thanks{An earlier version of this paper was presented at the 30th European Conference on Operational Research, Dublin, in June 2019, the International Congress on Industrial and Applied Mathematics, Valencia, in July 2019, and the International Conference on Nonlinear Analysis and Convex Analysis, Hakodate, in August 2019. The author is grateful to H\'el\`ene Frankowska, Sabine Pickenhain, Christiane Tammer, and an anonymous referee for their helpful comments.}}
\date{\today}
\author{Nobusumi Sagara\thanks{This research is supported by JSPS KAKENHI Grant Number JP18K01518 from the Ministry of Education, Culture, Sports, Science and Technology, Japan.}  \\[-3pt]
\\[-3pt]
{\small Faculty of Economics, Hosei University} \\[-4pt]
{\small 4342, Aihara, Machida, Tokyo, 194-0298, Japan} \\[-4pt]
{\footnotesize e-mail: nsagara@hosei.ac.jp}}

\begin{document}
\maketitle
\setcounter{page}{0}
\thispagestyle{empty}
\clearpage

\begin{abstract} 
This paper exemplifies that saturation is an indispensable structure on measure spaces to obtain the existence and characterization of solutions to nonconvex variational problems with integral constraints in Banach spaces and their dual spaces. We provide a characterization of optimality via the maximum principle for the Hamiltonian and an existence result without the purification of relaxed controls, in which the Lyapunov convexity theorem in infinite dimensions under the saturation hypothesis on the underlying measure space plays a crucial role. We also demonstrate that the existence of solutions for certain class of primitives is necessary and sufficient for the measure space to be saturated. \\


\end{abstract}

\section{Introduction}
Optimal resource allocation problems in mathematical economics are formulated as an isoperimetric problem of the following form:
\begin{equation}  
\label{P}
\begin{aligned}
& \inf_{f\in L^1(\mu,\R^n)}\int_T\varphi(t,f(t))d\mu \\
& \text{s.t. $\int_Tf(t)d\mu=x\in \R^n_+$ and $f(t)\in \R^n_+$ a.e.\ $t\in T$}
\end{aligned}
\tag{$\mathrm{P}$}
\end{equation}
where $(T,\Sigma,\mu)$ is a measure space and $\varphi:T\times \R^n_+\to \R$ is an integrand. The existence of solutions and the characterization of optimality in terms of the Euler--\hspace{0pt}Lagrange inequality for variational problem \eqref{P} were established in \cite{ap65} without any convexity assumptions. One of the prominent features of their result is the clarification of the role of the classical Lyapunov convexity theorem, which guarantees the convexity of the integral of a multifunction with values in finite-\hspace{0pt}dimensional vector spaces, a significant property to derive the necessary condition for optimality employing the separation theorem and to apply the relaxation technique for the existence. Subsequent works along this line are found in \cite{al72,ar74,ar80,ar19,bl73,cr98,fbjm14,hi74,io06,io09,it74,ma79,ma86,ol90}. 

Under the nonatomicity of measure spaces, the incorporation of integral constraints in Banach spaces into the above nonconvex variational problem is quite difficult for the existence and characterization of solutions because of the celebrated failure of the Lyapunov convexity theorem in infinite dimensions; see \cite{du77}. To overcome this difficulty, we strengthen the notion of nonatomicity and propose using saturated measure spaces originated in \cite{ma42} and then elaborated by \cite{fk02,hk84,ks09}. As shown in \cite{gp13,ks13,ks15,ks16,sa17}, the saturation of measure spaces is not only sufficient but also necessary for the Lyapunov convexity theorem to be true in separable Banach spaces and their dual spaces. An inevitable consequence of this fact is that it is futile to attempt to obtain a general existence result in infinite dimensions as long as one sticks to nonatomic measure spaces. 

The purpose of this paper is to exemplify that saturation is an indispensable structure on measure spaces to obtain the existence and characterization of solutions to nonconvex variational problems with integral constraints in separable Banach spaces and their dual spaces. The state constraints are formulated in the Bochner and Gelfand integral settings with control systems in separable metrizable spaces. The problem under consideration is a general reduced form of the isometric problem studied in \cite{ks14,sa15,sa17}, which is an infinite-\hspace{0pt}dimensional analogue of \cite{ap65} followed by the forementioned works. We provide a characterization of optimality via the subgradient like inequality for the integrand and the state constraint function, which describes the maximum principle for the Hamiltonian. Unlike \cite{ks14,sa17}, we prove an existence result without the purification of relaxed controls, in which the compactness and convexity of the integral of a multifunction with values in infinite dimensions under the saturation hypothesis on the underlying measure space along the lines of \cite{po08,sy08} play a crucial role. We also demonstrate that the existence of solutions for certain class of primitives is necessary and sufficient for the measure space to be saturated, which provides another characterization of saturation.

\section{Preliminaries}
\subsection{Bochner Integrals of Multifunctions}
Let $(T,\Sigma,\mu)$ be a complete finite measure space and $(E,\|\cdot\|)$ be a Banach space with its dual $E^*$ furnished with the dual system $\langle \cdot,\cdot \rangle$ on $E^*\times E$. A function $f:T\to E$ is \textit{strongly measurable} if there exists a sequence of simple (or finitely valued measurable) functions $f_n:T\to E$ such that $\|f(t)-f_n(t)\|\to 0$ a.e.\ $t\in T$; $f$ is \textit{Bochner integrable} if it is strongly measurable and $\int \|f(t)\|d\mu<+\infty$, where the \textit{Bochner integral} of $f$ over $A\in \Sigma$ is defined by $\int_A fd\mu=\lim_n\int_A f_nd\mu$. Denote by $L^1(\mu,E)$ the space of ($\mu$-equivalence classes of) $E$-\hspace{0pt}valued Bochner integrable functions on $T$, normed by $\| f \|_1=\int \| f(t) \|d\mu$, $f\in L^1(\mu,E)$. By the Pettis measurability theorem (see \cite[Theorem II.1.2]{du77}), $f$ is strongly measurable if and only if it is measurable with respect to the norm topology of $E$ whenever $E$ is separable.  

A set-\hspace{0pt}valued mapping $\Gamma$ from $T$ to the family of nonempty subsets of $E$ is called a \textit{multifunction}. Denote by $\overline{\mathrm{co}}\,\Gamma$ the multifunction defined by the closure of the convex hull of $\Gamma(t)$. The multifunction $\Gamma:T\twoheadrightarrow E$ is \textit{measurable} if the set $\{t\in T\mid \Gamma(t)\cap U\ne \emptyset \}$ is in $\Sigma$ for every open subset $U$ of $E$; it is \textit{graph measurable} if the set $\mathrm{gph}\,\Gamma:=\{ (t,x)\in T\times E\mid x\in \Gamma(t) \}$ belongs to $\Sigma\otimes \mathrm{Borel}(E,\| \cdot \|)$, where $\mathrm{Borel}(E,\| \cdot \|)$ is the Borel $\sigma$-\hspace{0pt}algebra of $(E,\| \cdot \|)$ generated by the norm topology. If $E$ is separable, then $\mathrm{Borel}(E,\| \cdot \|)$ coincides with the Borel $\sigma$-\hspace{0pt}algebra $\mathrm{Borel}(E,\mathit{w})$ of $E$ generated by the weak topology; see \cite[Part I, Chap.\,II, Corollary 2]{sc73}. It is well-\hspace{0pt}known that for closed-\hspace{0pt}valued multifunctions, measurability and graph measurability coincide whenever $E$ is separable; see \cite[Theorem III.30]{cv77}. A function $f:T\to E$ is a \textit{selector} of $\Gamma$ if $f(t)\in \Gamma(t)$ a.e.\ $t\in T$. If $E$ is separable, then by the Aumann measurable selection theorem, a multifunction $\Gamma$ with measurable graph admits a measurable selector (see \cite[Theorem III.22]{cv77}) and it is also strongly measurable. 

Let $B$ be the open unit ball in $E$. A multifunction $\Gamma:T\twoheadrightarrow E$ is \textit{integrably bounded} if there exists $\varphi\in L^1(\mu)$ such that $\Gamma(t)\subset \varphi(t)B$ a.e.\ $t\in T$. If $\Gamma$ is graph measurable and integrably bounded, then it admits a Bochner integrable selector whenever $E$ is separable. Denote by $\mathcal{S}^1_\Gamma$ the set of Bochner integrable selectors of $\Gamma$. The Bochner integral of $\Gamma$ is conventionally defined as $\int\Gamma d\mu:=\{ \int fd\mu \mid f\in \mathcal{S}^1_\Gamma \}$.

\subsection{Gelfand Integrals of Multifunctions}
A function $f:T\to E^*$ is \textit{weakly$^*\!$ scalarly measurable} if for every $x\in E$ the scalar function $\langle f(\cdot),x \rangle:T\to \R$ defined by $t\mapsto \langle f(t),x \rangle$ is measurable. Denote by $L^\infty(\mu,E^*_{\textit{w}^*})$ the space of ($\mu$-equivalence classes of) weakly$^*\!$ measurable, essentially bounded, $E^*$-valued functions on $T$, normed by $\| f \|_\infty=\mathrm{ess\,sup}_{t\in T}\| f(t) \|<\infty$. Then the dual space of $L^1(\mu,E)$ is given by $L^\infty(\mu,E^*_{\textit{w}^*})$ whenever $E$ is separable (see \cite[Theorem 2.112]{fl07}) and the dual system is given by $\langle f,g \rangle=\int\langle f(t),g(t) \rangle d\mu$ with $f\in L^\infty(\mu,E^*_{\mathit{w}^*})$ and $g\in L^1(\mu,E)$. Denote by $\mathrm{Borel}(E^*,\mathit{w}^*)$ the Borel $\sigma$-\hspace{0pt}algebra of $E^*$ generated by the weak$^*\!$ topology. If $E$ is a separable Banach space, then $E^*$ is separable with respect to the weak$^*\!$ topology. Hence, under the separability of $E$, a function $f:T\to E^*$ is weakly$^*\!$ scalarly measurable if and only if it is Borel measurable with respect to $\mathrm{Borel}(E^*,\mathit{w}^*)$; see \cite[Theorem 1]{th75}. Weakly$^*\!$ scalarly measurable functions $f_1,f_2:T\to E^*$ are \textit{weakly$^*\!$ scalarly equivalent} if $\langle f_1(t),x \rangle=\langle f_2(t),x \rangle$ for every $x\in E$ a.e.\ $t\in T$ (the exceptional $\mu$-\hspace{0pt}null set depending on $x$). 

A weakly$^*\!$ scalarly measurable function $f$ is \textit{weakly$^*\!$ scalarly integrable} if the scalar function $\langle f(\cdot),x \rangle$ is integrable for every $x\in E$. A weakly$^*$ scalarly measurable function $f$ is \textit{Gelfand integrable} over $A\in \Sigma$ if there exists $x^*_A\in E^*$ such that $\langle x^*_A,x \rangle=\int_A\langle f(t),x \rangle d\mu$ for every $x\in E$. The element $x^*_A$ is called the \textit{Gelfand integral} (or \textit{weak$^*\!$ integral}) of $f$ over $A$, denoted by $\mathit{w}^*\text{-}\int_Afd\mu$. Every weakly$^*\!$ scalarly integrable function is Gelfand integrable; see \cite[Theorem 11.52]{ab06}. Denote by $G^1(\mu,E^*)$ the space of equivalence classes of $E^*$-valued Gelfand integrable functions on $T$ with respect to weak$^*\!$ scalar equivalence, normed by $\| f \|_{\mathit{G}^1}=\sup_{x\in B}\int |\langle f(t),x \rangle| d\mu$. This norm is called the \textit{Gelfand norm} and the normed space $(G^1(\mu,E^*), \|\cdot \|_{\mathit{G}^1})$, in general, is not complete.

Let $\Gamma:T\twoheadrightarrow E^*$ be a multifunction. Denote by $\overline{\mathrm{co}}^{\mathit{\,w}^*}\Gamma:T\twoheadrightarrow E^*$ the multifunction defined by the weakly$^*\!$ closed convex hull of $\Gamma(t)$. A multifunction $\Gamma$ is \textit{measurable} if the set $\{t\in T\mid \Gamma(t)\cap U\ne \emptyset \}$ is in $\Sigma$ for every weakly$^*\!$ open subset $U$ of $E^*$. If $E$ is separable, then $E^*$ is a Suslin space, and hence, a multifunction $\Gamma$ with measurable graph in $\Sigma\otimes \mathrm{Borel}(E^*,\mathit{w}^*)$ admits a $\mathrm{Borel}(E^*,\mathit{w}^*)$-\hspace{0pt}measurable (or equivalently, weakly$^*\!$ measurable) selector; see \cite[Theorem III.22]{cv77}. 

Let $B^*$ be the open unit ball of $E^*$. A multifunction $\Gamma:T\twoheadrightarrow E^*$ is \textit{integrably bounded} if there exists $\varphi\in L^1(\mu)$ such that $\Gamma(t)\subset \varphi(t)B^*$ for every  $t\in T$. If $\Gamma$ is integrably bounded with measurable graph, then it admits a Gelfand integrable selector whenever $E$ is separable. Denote by $\mathcal{S}^{1,\mathit{w}^*}_\Gamma$ the set of Gelfand integrable selections of $\Gamma$. The Gelfand integral of $\Gamma$ is conventionally defined as $\mathit{w}^*\text{-}\int\Gamma d\mu:=\{ \mathit{w}^*\text{-}\int fd\mu \mid f\in \mathcal{S}^{1,\mathit{w}^*}_\Gamma \}$.

\subsection{Lyapunov Convexity Theorem in Banach Spaces}
In what follows, we always assume the completeness of $(T,\Sigma,\mu)$. Let $\Sigma_S=\{ A\cap S\mid A\in \Sigma \}$ be the $\sigma$-\hspace{0pt}algebra restricted to $S\in \Sigma$. Denote by $L^1_S(\mu)$ the space of $\mu$-\hspace{0pt}integrable functions on the measurable space $(S,\Sigma_S)$ whose elements are restrictions of functions in $L^1(\mu)$ to $S$. An equivalence relation $\sim$ on $\Sigma$ is given by $A\sim B \Leftrightarrow \mu(A\triangle B)=0$, where $A\triangle B$ is the symmetric difference of $A$ and $B$ in $\Sigma$. The collection of equivalence classes is denoted by $\Sigma(\mu)=\Sigma/\sim$ and its generic element $\widehat{A}$ is the equivalence class of $A\in \Sigma$. We define the metric $\rho$ on $\Sigma(\mu)$ by $\rho(\widehat{A},\widehat{B})=\mu(A\triangle B)$. Then $(\Sigma(\mu),\rho)$ is a complete metric space; see \cite[Lemma 13.13]{ab06}. When the equivalence relation $\sim$ is restricted to $\Sigma_S$, we obtain a metric space $(\Sigma_S(\mu),\rho)$ as a subspace of $(\Sigma(\mu),\rho)$. Note that $(\Sigma(\mu),\rho)$ is separable if and only if $L^1(\mu)$ is separable; see \cite[Lemma 13.14]{ab06}. The \textit{density} $\mathrm{dens}\,X$ of a topological space $X$ is the least cardinality of a dense subset of $X$. We denote by $\mathrm{dens}\,\Sigma(\mu)$ the density of the metric space $(\Sigma(\mu),\rho)$ induced by a finite measure space $(T,\Sigma,\mu)$. Then $L^1(\mu)$ is separable if and only if $\mathrm{dens}\,\Sigma(\mu)\le\aleph_0$. 

A finite measure space $(T,\Sigma,\mu)$ is said to be \textit{saturated} if $L^1_S(\mu)$ is nonseparable for every $S\in \Sigma$ with $\mu(S)>0$. Several equivalent definitions for saturation are known; see \cite{fk02,fr12,hk84,ks09}. One of the simple characterizations of the saturation property is as follows: $(T,\Sigma,\mu)$ is saturated if and only if $\mathrm{dens}\,\Sigma_S(\mu)\ge \aleph_1$ for every $S\in \Sigma$ with $\mu(S)>0$. Since $(T,\Sigma,\mu)$ is nonatomic if and only if $\mathrm{dens}\,\Sigma_S(\mu)\ge \aleph_0$ for every $S\in \Sigma$ with $\mu(S)>0$, saturation implies nonatomicity. An germinal notion of saturation in terms of the uncountability of $\mathrm{dens}\,\Sigma(\mu)$ already appeared in \cite{ka44,ma42}; see also \cite[331Y(e)]{fr12}. 

The following characterization of the saturation property will be useful for many applications.   

\begin{prop}
\label{lyp}
Let $(T,\Sigma,\mu)$ be a finite measure space and let $E$ be an infinite-\hspace{0pt}dimensional separable Banach space. Then the following conditions are equivalent.
\begin{enumerate}[\rm(i)]
\item $(T,\Sigma,\mu)$ is saturated. 
\item For every $\mu$-\hspace{0pt}continuous vector measure $m:\Sigma\to E$, its range $m(\Sigma)$ is weakly compact and convex in $E$.
\item For every $\mu$-\hspace{0pt}continuous vector measure $m:\Sigma\to E^*$, its range $m(\Sigma)$ is weakly$^*\!$ compact and convex in $E^*$.
\item For every multifunction $\Gamma:T\twoheadrightarrow E$, its Bochner integral $\int\Gamma d\mu$ is convex.
\item For every integrably bounded, weakly compact-\hspace{0pt}valued multifunction $\Gamma:T\twoheadrightarrow E$ with the measurable graph, $\int\Gamma d\mu=\int\overline{\mathrm{co}}\,\Gamma d\mu$.
\item For every multifunction $\Gamma:T\twoheadrightarrow E^*$, its Gelfand integral $\mathit{w}^*\text{-}\int\Gamma d\mu$ is convex.
\item For every integrably bounded, weakly$^*\!$ compact-\hspace{0pt}valued multifunction $\Gamma:T\twoheadrightarrow E^*$ with the measurable graph, $\mathit{w}^*\text{-}\int\Gamma d\mu=\mathit{w}^*\text{-}\int\overline{\mathrm{co}}^{\,\mathit{w}^*}\Gamma d\mu$.
\end{enumerate}
In particular, the implications (i) $\Rightarrow$ (ii), (iii), (iv), (v) are true for every separable Banach space. 
\end{prop}

\begin{rem}
The equivalence (i) $\Leftrightarrow$ (ii) is proven by \cite{ks13}, (i) $\Leftrightarrow$ (iii) is established by \cite{gp13}, and (i) $\Leftrightarrow$ (iv) $\Leftrightarrow$ (v) $\Leftrightarrow$ (vi) $\Leftrightarrow$ (vii) is due to \cite{po08,sy08}. For the Lyapunov convexity theorem in non-\hspace{0pt}separable locally convex Hausdorff spaces to be valid, $\mathrm{dens}\,\Sigma_S(\mu)$ needs to be greater than $\mathrm{dens}\,E$ for every $S\in \Sigma$ with $\mu(S)>0$; see for the detail \cite{gp13,ks15,ks16,ur19}. 
\end{rem}

\section{Nonconvex Variational Problems with Integral Constraints}
\subsection{The Case for Bochner Integrals}
Let $\overline{\R}=\R\cup \{ \pm\infty \}$ be the extended real line. The \textit{domain} of an extended real-valued function $V:E\to \overline{\R}$ is the set given by $\mathrm{dom}\,V=\{ x\in E\mid V(x)<+\infty \}$. A function $V$ is said to be \textit{proper} if $\mathrm{dom}\,V$ is nonempty and $V(x)>-\infty$ for every $x\in E$. The \textit{subdifferential} of $V$ at $x\in E$ with $V(x)\in \R$ is the set given by $\partial V(x)=\{ x^*\in E^*\mid \langle x^*,y-x \rangle \le V(y)-V(x) \ \forall y\in E \}$, where we set $\partial V(x)=\emptyset$ whenever $V(x)\not\in \R$. An element in $\partial V(x)$ is called a \textit{subgradient} of $V$ at $x$. For a convex subset $C$ of $E$, the \textit{normal cone} of $C$ at $x\in C$ is the set $N_C(x)=\{ x^*\in E^*\mid \langle x^*,y-x \rangle\le 0 \ \forall y\in C \}$.   

Let $X$ be a topological space endowed with its Borel $\sigma$-algebra $\mathrm{Borel}\,(X)$. A function $\Phi:T\times X\to E$ is called a \textit{Carath\'eodory function} if $t\mapsto \Phi(t,a)$ is measurable for every $a\in X$ and $a\mapsto \Phi(t,a)$ is continuous for every $t\in T$. If $X$ is a separable metrizable space, then the Carath\'eodory function $\Phi$ is $\Sigma\otimes \mathrm{Borel}\,(X)$-measurable; see \cite[Lemma 4.51]{ab06}. A real-valued Carath\'eodory function $\varphi:T\times X\to \R$ is called a \textit{Carath\'eodory integrand}. 

Denote by $\M(T,X)$ the space of measurable functions $u:T\to X$ and let $U:T\twoheadrightarrow X$ be a multifunction. The variational problem under consideration is of the form: 
\begin{equation}
\label{P_0}
\begin{aligned}
& \inf_{u\in \M(T,X)} \int_T\varphi(t,u(t))d\mu \quad  \\
& \text{s.t. $\int_T \Phi(t,u(t))d\mu\in C$ and $u(t)\in U(t)$ a.e.\ $t\in T$}
\end{aligned}
\tag{$\mathrm{P}_0$}
\end{equation}
where the hypotheses on the primitive $(\varphi,\Phi,U,C)$ are given in order below. The set of feasible controls is defined by
$$
\U:=\left\{ u\in \M(T,X) \mid \int_T \Phi(t,u(t))d\mu\in C,\,u(t)\in U(t)\ \text{a.e.\ $t\in T$} \right\}. 
$$
Then $\hat{u}\in \U$ is called a \textit{solution} to problem \eqref{P_0} whenever $\int\varphi(t,\hat{u}(t))d\mu=\inf_{u\in \U}\int\varphi(t,u(t))d\mu\in \R$. Define the \textit{Hamiltonian} $H:T\times E^*\to \R\cup \{ +\infty \}$ by 
$$
H(t,x^*):=\sup_{a\in U(t)}\left\{ \left\langle x^*,\Phi(t,a) \right\rangle-\varphi(t,a) \right\}. 
$$
We say that $u\in \U$ satisfies the \textit{maximum principle} for an adjoint variable $x^*\in E^*$ if  
\begin{equation}
\label{mp}
H(t,x^*)=\langle x^*,\Phi(t,u(t))\rangle-\varphi(t,u(t)) \quad\text{a.e.\ $t\in T$}. \tag{$\mathrm{MP}$}
\end{equation}

For a given $x\in E$, consider the perturbed problem for \eqref{P_0}: 
\begin{equation}
\label{P_x}
\begin{aligned}
& \inf_{u\in \M(T,X)} \int_T\varphi(t,u(t))d\mu \quad  \\
& \text{s.t. $\int_T \Phi(t,u(t))d\mu\in C+x$ and $u(t)\in U(t)$ a.e.\ $t\in T$}.
\end{aligned}
\tag{$\mathrm{P}_x$}
\end{equation}
It is obvious that problem \eqref{P_0} coincides with perturbed problem \eqref{P_x} with $x=0$. The set of perturbed feasible controls is given by
$$
\U_x:=\left\{ u\in \M(T,X) \mid \int_T \Phi(t,u(t))d\mu\in C+x,\,u(t)\in U(t)\ \text{a.e.\ $t\in T$} \right\}. 
$$
Denote the infimum value of \eqref{P_x} by
\begin{align*}
V(x):=
\begin{cases}
\displaystyle\inf_{u\in \U_x}\int_T\varphi(t,u(t))d\mu & \text{if $\U_x\ne \emptyset$}, \\
+\infty & \text{otherwise}.
\end{cases}
\end{align*}
Then $x\mapsto V(x)$ defines the \textit{value function} $V:E\to \overline{\R}$ with $V(0)=\inf_{u\in \U}\int\varphi(t,u(t))d\mu$ whenever $\U$ is nonempty.  

Throughout the rest of the paper we always assume that $E$ is a separable Banach space and $X$ is a separable metrizable space. The standing hypotheses on the primitive are given as follows.  

\begin{description}
\item[$\mathbf{(H_1)}$]  $\varphi:T\times X\to \R$ is a Carath\'eodory integrand;
\item[$\mathbf{(H_2)}$]  $\Phi:T\times X\to E$ is a Carath\'eodory function;
\item[$\mathbf{(H_3)}$]  $U:T\twoheadrightarrow X$ is a multifunction with $\mathrm{gph}\,U\in \Sigma\otimes \mathrm{Borel}(X)$.  
\end{description}

The relation between the adjoint inclusion and the maximum principle is formulated as follows. 

\begin{thm}
\label{thm1}
Let $C$ be a nonempty convex subset of $E$. Under hypotheses $\mathrm{(H_1)}$--\hspace{0pt}$\mathrm{(H_3)}$, the following conditions are equivalent.
\begin{enumerate}[\rm(i)]
\item $x^*\in \partial V(0)$ and $u$ is a solution to \eqref{P_0}. 
\item $-x^*\in N_C\left( \displaystyle\int_T\Phi(t,u(t))d\mu \right)$ and $u$ satisfies \eqref{mp}. 
\end{enumerate}
\end{thm}

\begin{proof}
(i) $\Rightarrow$ (ii): The subgradient inequality for $V$ yields 
\begin{equation}
\label{eq1}
\langle x^*,x \rangle\le V(x)-V(0) \quad\text{for every $x\in E$}.
\end{equation} 
Take any $y\in C$ and set $x=\int\Phi(t,u(t))d\mu-y$. Since $u$ is a solution to \eqref{P_0} and $\int\Phi(t,u(t))d\mu\in C+x$, we have $V(x)\le \int\varphi(t,u(t))d\mu=V(0)$. The inclusion $-x^*\in N_C(\int\Phi(t,u(t))d\mu)$ follows from inequality \eqref{eq1} for this choice of $x$. To prove \eqref{mp}, choose any $v\in \M(T,X)$ with $v(t)\in U(t)$ a.e.\ $t\in T$ and set $x=\int\Phi(t,v(t))d\mu-\int\Phi(t,u(t))d\mu$. Since $\int\Phi(t,v(t))d\mu\in C+x$, we have $V(x)\le \int\varphi(t,v(t))d\mu$. It thus follows from \eqref{eq1} that 
\begin{equation}
\label{eq2}
\begin{aligned}
& \left\langle x^*,\int_T\Phi(t,v(t))d\mu-\int_T\Phi(t,u(t))d\mu \right\rangle  \\ {}\le{} 
& \int_T\varphi(t,v(t))d\mu-\int_T\varphi(t,u(t))d\mu. 
\end{aligned}
\end{equation}
To derive \eqref{mp} from \eqref{eq2}, assume that there exists $A\in \Sigma$ of positive measure such that for every $t\in A$ there exists $a\in U(t)$ satisfying $H(t,x^*)\ge \left\langle x^*,\Phi(t,a) \right\rangle-\varphi(t,a)>\left\langle x^*,\Phi(t,u(t)) \right\rangle-\varphi(t,u(t))$. Define the multifunction $\Psi:A\twoheadrightarrow \R_{++}\times X$ by 
\begin{align*}
& \Psi(t):=\left\{ (r,a)\in \R_{++}\times U(t)\,\left| \begin{array}{l}\langle x^*,\Phi(t,u(t)) \rangle-\varphi(t,u(t))+r \\ \quad =\langle x^*,\Phi(t,a) \rangle-\varphi(t,a) \end{array} \right.\right\}.
\end{align*}
Since $\mathrm{gph}\,\Psi\in \Sigma_A\otimes \mathrm{Borel}\,(\R_{++})\otimes \mathrm{Borel}\,(X)$ (see \cite[Corollary 18.8]{ab06}), $\Psi$ admits a measurable selector. Let $u_A:A\to X$ be the $X$-\hspace{0pt}component mapping of a measurable selector from $\Psi$. Define $v\in \M(T,X)$ by $v(t)=u_A(t)$ on $A$ and $v(t)=u(t)$ on $T\setminus A$. By construction, $v(t)\in U(t)$ a.e.\ $t\in T$, and $\langle x^*,\Phi(t,u(t)) \rangle-\varphi(t,u(t))<\langle x^*,\Phi(t,v(t)) \rangle-\varphi(t,v(t))$ on $A$ and $\langle x^*,\Phi(t,u(t)) \rangle-\varphi(t,u(t))=\langle x^*,\Phi(t,a) \rangle-\varphi(t,a)$ on $T\setminus A$. Integrating the inequality over $A$ and the equality over $T\setminus A$ and adding both sides of the integrals, we derive a contradiction to \eqref{eq2}. 

(ii) $\Rightarrow$ (i): Conversely, suppose that $-x^*\in N_{C}(\int\Phi(t,u(t))d\mu)$ and $u$ satisfies \eqref{mp}. Choose any $v\in \U$. Integrating both sides of the inequality $\varphi(t,v(t))-\varphi(t,u(t))\ge \langle x^*,\Phi(t,v(t))-\Phi(t,u(t)) \rangle$ that stems from \eqref{mp} yields inequality \eqref{eq1} with $\langle x^*,\int\Phi(t,v(t))d\mu-\int\Phi(t,u(t))d\mu \rangle\ge 0$, which follows from the fact that $\int\Phi(t,v(t))d\mu\in C$ and the definition of the normal cone $N_{C}(\int\Phi(t,u(t))d\mu)$. Hence, $u$ is a solution to \eqref{P_0}. To prove that $x^*\in \partial V(0)$, take any $x\in E$ such that $\U_x$ is nonempty. Since $\int\Phi(t,v(t))d\mu-x\in C$ for every $v\in \U_x$, we have $\left\langle x^*,\int\Phi(t,v(t))d\mu-x-\int\Phi(t,u(t))d\mu \right\rangle \ge 0$. It follows from \eqref{eq2} that $\int\varphi(t,v(t))d\mu-\int\varphi(t,u(t))d\mu\ge \langle x^*,x \rangle$ for every $v\in \U_x$. Therefore, we obtain $V(x)-V(0)\ge \langle x^*,x \rangle$ for every $x\in E$ with $\U_x\ne \emptyset$. 
\end{proof}

Define the multifunction $\Gamma_{\varphi,\Phi}^U:T\twoheadrightarrow \R\times E$ by 
$$
\Gamma_{\varphi,\Phi}^U(t):=\bigcup\left\{ (\varphi(t,a),\Phi(t,a)) \mid a\in U(t) \right\}.
$$
Namely, $\Gamma_{\varphi,\Phi}^U$ is obtained as the image of the set $U(t)$ under the vector-valued function $(\varphi(t,\cdot),\Phi(t,\cdot,)):X\to \R\times E$ for every $t\in T$. To employ the separation theorem in Banach spaces, we impose the following hypotheses which is peculiar to infinite dimensions. 
 
\begin{description}
\item[$\mathbf{(H_4)}$] Either $\displaystyle\left( \mathrm{int}\int_T\Gamma_{\varphi,\Phi}^U(t)d\mu \right)\cap (\R\times C)$ or $\displaystyle\int_T\Phi(t,U(t))d\mu\cap \mathrm{int}\,C$ is non\-empty. 
\end{description}
There is a trade-\hspace{0pt}off between the choice of the nonempty interior conditions. Although the nonemptiness of $\mathrm{int}\int\Gamma_{\varphi,\Phi}^Ud\mu$ seems stringent, we can deal with the situation in which $C$ is a singleton as in the resource allocation problem \eqref{P}. At the cost of the nonsingleton of $C$, the alternative hypothesis that $\mathrm{int}\,C$ is nonempty might be an innocuous condition in many applications. Note that, however,  whenever $C$ is a positive cone of an ordered Banach space, the class of such spaces with nonempty $\mathrm{int}\,C$ is very limited. For instance, the positive cone of $L^p$ spaces has an empty interior for $1\le p<\infty$. Typical separable Banach lattices such that the positive cone has a nonempty interior is the space of continuous functions on a compact metrizable space endowed with the sup norm topology. 

\begin{thm}
\label{thm2}
Let $(T,\Sigma,\mu)$ be a saturated finite measure space and $C$ be a nonempty convex subset of $E$. Under hypotheses $\mathrm{(H_1)}$--\hspace{0pt}$\mathrm{(H_4)}$, $u\in \U$ is a solution to \eqref{P_0} if and only if there exists 
$$
-x^*\in N_C\left( \int_T\Phi(t,u(t))d\mu \right)
$$ 
such that \eqref{mp} holds. 
\end{thm}

\begin{proof}
Suppose that $u\in \U$ is a solution to \eqref{P_0}. In view of the proof of Theorem \ref{thm1}, it suffices to show that there exists $-x^*\in N_{C}(\int\Phi(t,u(t))d\mu)$ such that inequality \eqref{eq2} holds for every $v\in \M(T,X)$ with $v(t)\in U(t)$ a.e.\ $t\in T$. Toward this end, define the convex subset $D$ of $\R\times E$ by $D:=\{ r\in \R\mid r<\int\varphi(t,u(t))d\mu \}\times C$. Note that either $\mathrm{int}\int\Gamma_{\varphi,\Phi}^U d\mu$ or $\mathrm{int}\,D$ is nonempty in view of $\mathrm{(H_4)}$ and $\int\Gamma_{\varphi,\Phi}^U d\mu$ is convex because of Proposition \ref{lyp}. Since $u$ is a solution to \eqref{P_0}, we have $D\cap \int\Gamma_{\varphi,\Phi}^U d\mu=\emptyset$. By the separation theorem (see \cite[Theorem 5.67]{ab06}), there exists $(\alpha,y^*)\in (\R\times E^*)\setminus \{ 0 \}$ that separates the closure $\overline{D}$ of $D$ and $\int\Gamma_{\varphi,\Phi}^U d\mu$, that is, 
\begin{equation}
\label{eq4}
\alpha r+\langle y^*,x \rangle\le \alpha s+\langle y^*,y \rangle
\end{equation}
for every $(r,x)\in \overline{D}$ and $(s,y)\in \int\Gamma_{\varphi,\Phi}^U d\mu$. It follows from the inclusion 
$$
\int_T\left( \varphi(t,u(t)),\Phi(t,u(t)) \right)d\mu\in \overline{D}\cap \int_T\Gamma_{\varphi,\Phi}^U(t)d\mu
$$ 
that $\langle y^*,x \rangle\le \langle y^*,\int\Phi(t,u(t))d\mu \rangle$ for every $x\in C$. This means that $y^*\in N_C(\int\Phi(t,u(t))d\mu)$. It follows from \eqref{eq4} that $\alpha\ge 0$. Assume that $\alpha=0$. We then have $\langle y^*,x \rangle\le \langle y^*,y \rangle$ for every $x\in C$ and $y\in \int\Phi(t,U(t))d\mu$. If $(\mathrm{int}\int\Gamma_{\varphi,\Phi}^U d\mu)\cap (\R\times C)$ is nonempty, then there is $y_0\in (\mathrm{int}\int\Phi(t,U(t))d\mu)\cap C$ such that for every $h\in E$ we have $y_0\pm \varepsilon h\in \mathrm{int}\int\Phi(t,U(t))d\mu$ for all sufficiently small $\varepsilon>0$, and hence, $\langle y^*,x \rangle\le \langle y^*,y_0\pm\varepsilon h \rangle$ for every $x\in C$. Letting $x=y_0$ inevitably yields $y^*=0$, a contradiction. If $\int\Phi(t,U(t))d\mu\cap\mathrm{int}\,C$ is nonempty, then there exists $x_0\in\int\Phi(t,U(t))d\mu\cap \mathrm{int}\,C$ such that for every $h\in E$ we have $x_0\pm \varepsilon h\in \mathrm{int}\,C$ for all sufficiently small $\varepsilon>0$, and hence, $\langle y^*,x_0\pm\varepsilon h \rangle\le \langle y^*,y \rangle$ for every $y\in \int\Phi(t,U(t))d\mu$. Letting $y=x_0$ inevitably yields $y^*=0$, a contradiction again. Henceforth, $\alpha>0$. The normalization $x^*=-y^*/\alpha$ in \eqref{eq4} implies inequality \eqref{eq2} with $-x^*\in  N_{C}(\int\Phi(t,u(t))d\mu)$. 

The converse implication is evident from Theorem \ref{thm1}.  
\end{proof}

\begin{description}
\item[$\mathbf{(H_4')}$] $\displaystyle\int_T\Phi(t,U(t))d\mu\cap C$ is nonempty. 

\item[$\mathbf{(H_5)}$] $\Gamma_{\varphi,\Phi}^U:T\twoheadrightarrow \R\times E$ is integrably bounded with weakly compact values. 
\end{description}

The next result characterizes the saturation of measure spaces in terms of the existence of solutions to \eqref{P_0}.  

\begin{thm}
\label{thm3}
Let $(T,\Sigma,\mu)$ be a nonatomic finite measure space and $C$ be a nonempty weakly closed subset of $E$. If $(T,\Sigma,\mu)$ is saturated, then for every quartet $(\varphi,\Phi,U,C)$ satisfying hypotheses $\mathrm{(H_1)}$--\hspace{0pt}$\mathrm{(H_3)}$, $\mathrm{(H_4')}$, and $\mathrm{(H_5)}$, a solution to \eqref{P_0} exists. Conversely, if a solution to \eqref{P_0} exists for every quartet $(\varphi,\Phi,U,C)$ satisfying $\mathrm{(H_1)}$--\hspace{0pt}$\mathrm{(H_3)}$, $\mathrm{(H_4')}$, and $\mathrm{(H_5)}$, then $(T,\Sigma,\mu)$ is saturated whenever $E$ is infinite dimensional and $X$ is an uncountable compact Polish space.  
\end{thm}

\begin{proof}
Let $\{ u_n \}_{n\in \N}\subset \M(T,X)$ be a minimizing sequence of \eqref{P_0} satisfying $\int \varphi(t,u_n(t))d\mu\to \inf_{u\in \U}\int\varphi(t,u(t))d\mu$ with $\int\Phi(t,u_n(t))d\mu\in C$ and $u_n(t)\in U(t)$ a.e.\ $t\in T$ for each $n\in \N$. Since it follows from Proposition \ref{lyp} that the Bochner integral $\int\Gamma_{\varphi,\Phi}^Ud\mu$ is a weakly compact set containing the sequence $\{ \int(\varphi(t,u_n(t)),\Phi(t,u_n(t)))d\mu \}_{n\in \N}$, there exists a subsequence of it (which we do not relabel) such that $(r,x):=\mathit{w}\text{-}\lim_n\int (\varphi(t,u_n(t)),\Phi(t,u_n(t)))d\mu\in \int\Gamma_{\varphi,\Phi}^Ud\mu$ by the Eberlein--\hspace{0pt}\v{S}mulian theorem (the coincidence of weak compactness and sequential weak compactness); see \cite[Theorem 6.34]{ab06}. By virtue of the Filippov implicit function theorem (see \cite[Theorem 18.17]{ab06}), there exists $\hat{u}\in \M(T,X)$ with $\hat{u}(t)\in U(t)$ a.e.\ $t\in T$ such that $(r,x)=\int(\varphi(t,\hat{u}(t)),\Phi(t,\hat{u}(t)))d\mu$. We thus obtain 
$$
\int_T\varphi(t,\hat{u}(t))d\mu=\lim_{n\to\infty}\int_T\varphi(t,u_n(t))d\mu=\inf_{u\in \U}\int_T\varphi(t,u(t))d\mu
$$ 
and 
$$
\int_T\Phi(t,\hat{u}(t))d\mu=\mathit{w}\text{-}\lim_{n\to \infty}\int_T\Phi(t,u_n(t))d\mu\in C.
$$
Therefore, $\hat{u}$ is a solution to \eqref{P_0}. 

The converse implication is demonstrated in \cite[Theorem 4.4]{sa17} for the dual space setting with Gelfand integrals, but the proof per se is valid for the Bochner integral case. 
\end{proof}

The nonemptiness of $\partial V(x)$ at every point $x$ in the interior of $\mathrm{dom}\,V$ is guaranteed in the following result. 

\begin{thm}
\label{thm4}
Let $(T,\Sigma,\mu)$ be a saturated finite measure space and $C$ be a nonempty, closed, convex subset of $E$. Under hypotheses $\mathrm{(H_1)}$--\hspace{0pt}$\mathrm{(H_3)}$, $\mathrm{(H_4')}$, and $\mathrm{(H_5)}$, $V$ is a proper convex function that is continuous on $\mathrm{int}(\mathrm{dom}\,V)$. Furthermore, $V$ is weakly lower semicontinuous on $E$.  
\end{thm}

\begin{proof}
Since the Bochner integral $\int\Gamma_{\varphi,\Phi}^U d\mu$ is weakly compact and convex by Proposition \ref{lyp}, the convexity of $V$ follows from the equality
$$
V(x)=\min\left\{ r\in \R\mid (r,y)\in \int_T\Gamma_{\varphi,\Phi}^U(t)d\mu\cap \left( \R\times (C+x) \right) \right\}. 
$$
The properness of $V$ follows from the fact that $\Gamma_{\varphi,\Phi}^U$ is integrably bounded and $0\in \mathrm{dom}\,V$. Since $V$ is bounded from above on $\mathrm{dom}\,V$, it is continuous on $\mathrm{int}(\mathrm{dom}\,V)$; see \cite[Theorem 3.2.1]{it74}.

To demonstrate the weak lower semicontinuity of $V$, it suffices to show that $\mathrm{epi}\,V=\{ (r,x)\in \R\times E\mid V(x)\le r \}$ is weakly closed. Toward this end, let $\{ (r_\alpha,x_\alpha) \}_{\alpha\in \Lambda}$ be a net in $\mathrm{epi}\,V$ that converges weakly to $(r,x)\in \R\times E$. Assume, by way of contradiction, that $(r,x)$ does not belong to $\mathrm{epi}\,V$. We then have $\lim_\alpha r_\alpha=r<V(x)$. Since $V(x_\alpha)\le r_\alpha$, extracting a subnet from $\{ x_\alpha \}_{\alpha\in \Lambda}$ (which we do not relabel) yields $\U_{x_\alpha}\ne \emptyset$ for each $\alpha\in \Lambda$. Replacing $C$ by $C+x_\alpha$ in Theorem \ref{thm3} guarantees that there exists a net $\{ u_\alpha \}_{\alpha\in \Lambda}$ in $\M(T,X)$ such that $V(x_\alpha)=\int\varphi(t,u_\alpha(t))d\mu$, $\int\Phi(t,u_\alpha(t))d\mu\in C+x_\alpha$, and $u_\alpha(t)\in U(t)$ for each $\alpha\in \Lambda$ a.e.\ $t\in T$. Since the weakly compact set $\int\Gamma_{\varphi,\Phi}^Ud\mu$ contains the net $\{ \int(\varphi(t,u_\alpha(t)),\Phi(t,u_\alpha(t)))d\mu \}_{\alpha\in \Lambda}$, there exists a subnet of it (which we do not relabel) such that 
$$
(s,y):=\mathit{w}\text{-}\lim_\alpha\int_T(\varphi(t,u_\alpha(t)),\Phi(t,u_\alpha(t)))d\mu\in \int_T\Gamma_{\varphi,\Phi}^U(t)d\mu.
$$
By virtue of the Filippov implicit function theorem, there exists $u\in \M(T,X)$ with $u(t)\in U(t)$ a.e.\ $t\in T$ such that $(s,y)=\int(\varphi(t,u(t)),\Phi(t,u(t)))d\mu$. Since $\lim_{\alpha}\int\varphi(t,u_\alpha(t))d\mu=\int\varphi(t,u(t))d\mu$ and $$
\mathit{w}\text{-}\lim_{\alpha}\int_T\Phi(t,u_\alpha(t))d\mu=\int_T\Phi(t,u(t))d\mu\in C+x,
$$
we obtain $r<V(x)\le \int\varphi(t,u(t))d\mu=\lim_\alpha V(x_\alpha)\le \lim_{\alpha}r_\alpha=r$, a contradiction. This means that $\mathrm{epi}\,V$ is weakly closed. 
\end{proof}

\subsection{The Case for Gelfand Integrals}
Replacing the Bochner integral constraint in \eqref{P_0} by the Gelfand integral one in the dual space setting enables us to treat dual spaces of a separable Banach space. Typical dual spaces encompassed in this framework are $L^\infty$ spaces with a countably generated $\sigma$-\hspace{0pt}algebra (the dual spaces of a separable $L^1$ space) and the space of Radon measures on a compact metrizable space (the dual space of the space of continuous functions on a compact metrizable space). 

The \textit{subdifferential} of the extended real-valued function $V:E^*\to \overline{\R}$ at $x^*\in E^*$ is the set given by $\partial V(x^*)=\{ x\in E\mid \langle y^*-x^*,x \rangle\le V(y^*)-V(x^*) \ \forall y^*\in E^*\}$, where we set $\partial V(x^*)=\emptyset$ whenever $V(x^*)\not\in \R$. An element in $\partial V(x^*)$ is called a \textit{subgradient} of $V$ at $x^*$. For a convex subset $C$ of $E^*$, the \textit{normal cone} of $C$ at $x^*\in C$ is the set $N_C(x^*)=\{ x\in E\mid \langle y^*-x^*,x \rangle\le 0 \ \forall y^*\in C \}$.   

Let $\Phi:T\times X\to E^*$ be a measurable function (here the measurability is with respect to $\mathrm{Borel}\,(E^*,\| \cdot \|)$ and $C$ be a subset of $E^*$. The variational problem under investigation is
\begin{equation}
\label{P^*_0}
\begin{aligned}
& \inf_{u\in \M(T,X)} \int_T\varphi(t,u(t))d\mu \quad  \\
& \text{s.t. $\mathit{w}^*\text{-}\int_T \Phi(t,u(t))d\mu\in C$ and $u(t)\in U(t)$ a.e.\ $t\in T$}.
\end{aligned}
\tag{$\mathrm{P}^*_0$}
\end{equation}
The set of feasible controls is defined by
$$
\U:=\left\{ u\in \M(T,X) \mid \mathit{w}^*\text{-}\int_T \Phi(t,u(t))d\mu\in C,\,u(t)\in U(t)\ \text{a.e.\ $t\in T$} \right\}. 
$$
Define the \textit{Hamiltonian} $H:T\times E\to \R\cup \{ +\infty \}$ by 
$$
H(t,x):=\sup_{a\in U(t)}\left\{ \left\langle \Phi(t,a),x \right\rangle-\varphi(t,a) \right\}. 
$$
We say that $u\in \U$ satisfies the \textit{maximum principle} for an adjoint variable $x\in E$ if  
\begin{equation}
\label{mp*}
H(t,x)=\left\langle \Phi(t,u(t)),x \right\rangle-\varphi(t,u(t)) \quad\text{a.e.\ $t\in T$}. \tag{$\mathrm{MP}^*$}
\end{equation}

For a given $x^*\in E^*$, consider the perturbed problem for \eqref{P^*_0}:
\begin{equation}
\label{P^*_x^*}
\begin{aligned}
& \inf_{u\in \M(T,X)} \int_T\varphi(t,u(t))d\mu \quad  \\
& \text{s.t. $\mathit{w}^*\text{-}\int_T \Phi(t,u(t))d\mu\in C+x^*$ and $u(t)\in U(t)$ a.e.\ $t\in T$}.
\end{aligned}
\tag{$\mathrm{P}^*_{x^*}$}
\end{equation}
It is obvious that problem \eqref{P^*_0} coincides with perturbed problem \eqref{P^*_x^*} with $x^*=0$. The set of perturbed feasible controls is given by
$$
\U_{x^*}:=\left\{ u\in \M(T,X) \mid \mathit{w}^*\text{-}\int_T \Phi(t,u(t))d\mu\in C+x^*,\,u(t)\in U(t)\ \text{a.e.\ $t\in T$} \right\}. 
$$
Denote the optimum value of the perturbed problem of \eqref{P^*_0} by
\begin{align*}
V(x^*):=
\begin{cases}
\displaystyle\inf_{u\in \U_{x^*}}\int_T\varphi(t,u(t))d\mu & \text{if $\U_{x^*}\ne \emptyset$}, \\
+\infty & \text{otherwise}.
\end{cases}
\end{align*}
Then $x^*\mapsto V(x^*)$ defines the \textit{value function} $V:E^*\to \overline{\R}$ with $V(0)=\inf_{u\in \U}\int\varphi(t,u(t))d\mu$ whenever $\U$ is nonempty.  

For the dual space setting, hypothesis $\mathrm{(H_2)}$ is replaced by: 

\begin{description}
\item[$\mathbf{(H_2^*)}$]  $\Phi:T\times X\to E^*$ is a Carath\'eodory function.
\end{description}

\begin{thm}
\label{thm5}
Let $C$ be a nonempty convex subset of $E^*$. Under hypotheses $\mathrm{(H_1)}$, $\mathrm{(H_2^*)}$, and $\mathrm{(H_3)}$, the following conditions are equivalent.
\begin{enumerate}[\rm(i)]
\item $x\in \partial V(0)$ and $u$ is a solution to \eqref{P^*_0}. 
\item $-x\in N_C\left( \displaystyle\mathit{w}^*\text{-}\int_T\Phi(t,u(t))d\mu \right)$ and $u$ satisfies \eqref{mp*}. 
\end{enumerate}
\end{thm}

\begin{proof}
(i) $\Rightarrow$ (ii): The subgradient inequality for $v$ yields 
\begin{equation}
\label{eq5}
\langle x^*,x \rangle\le V(x^*)-V(0) \quad\text{for every $x^*\in E^*$}.
\end{equation} 
Take any $y^*\in C$ and set $x^*=\mathit{w}^*\text{-}\int\Phi(t,u(t))d\mu-y^*$. Since $u$ is a solution to \eqref{P^*_0} and $\mathit{w}^*\text{-}\int\Phi(t,u(t))d\mu\in C+x^*$, we have $V(x^*)\le \int\varphi(t,u(t))d\mu=V(0)$. The inclusion $-x\in N_C(\mathit{w}^*\text{-}\int\Phi(t,u(t))d\mu)$ follows from inequality \eqref{eq5} for this choice of $x^*$. To prove \eqref{mp*}, choose any $v\in \M(T,X)$ with $v(t)\in U(t)$ a.e.\ $t\in T$ and set $x^*=\mathit{w}^*\text{-}\int\Phi(t,v(t))d\mu-\mathit{w}^*\text{-}\int\Phi(t,u(t))d\mu$. Since $\mathit{w}^*\text{-}\int\Phi(t,v(t))d\mu\in C+x^*$, we have $V(x^*)\le \int\varphi(t,v(t))d\mu$. It thus follows from \eqref{eq5} that 
\begin{equation}
\label{eq6}
\begin{aligned}
& \left\langle \mathit{w}^*\text{-}\int_T\Phi(t,v(t))d\mu-\mathit{w}^*\text{-}\int_T\Phi(t,u(t))d\mu,x \right\rangle  \\ {}\le{} 
& \int_T\varphi(t,v(t))d\mu-\int_T\varphi(t,u(t))d\mu.
\end{aligned}
\end{equation}
To derive \eqref{mp*} from \eqref{eq6}, assume that there exists $A\in \Sigma$ of positive measure such that for every $t\in A$ there exists $a\in U(t)$ satisfying $H(t,x)\ge \left\langle \Phi(t,a),x \right\rangle-\varphi(t,a)>\left\langle \Phi(t,u(t)),x \right\rangle-\varphi(t,u(t))$. Define the multifunction $\Psi:A\twoheadrightarrow \R_{++}\times X$ by 
\begin{align*}
& \Psi(t):=\left\{ (r,a)\in \R_{++}\times U(t)\,\left| \begin{array}{l}\langle \Phi(t,u(t)),x \rangle-\varphi(t,u(t))+r \\ \quad =\langle \Phi(t,a),x \rangle-\varphi(t,a) \end{array} \right.\right\}.
\end{align*}
Since $\mathrm{gph}\,\Psi\in \Sigma_A\otimes \mathrm{Borel}\,(\R_{++})\otimes \mathrm{Borel}\,(X)$ (see \cite[Corollary 18.18]{ab06}), $\Psi$ admits a measurable selector. Let $u_A:A\to X$ be the $X$-\hspace{0pt}component mapping of a measurable selector from $\Psi$. Define $v\in \M(T,X)$ by $v(t)=u_A(t)$ on $A$ and $v(t)=u(t)$ on $T\setminus A$. By construction, $v(t)\in U(t)$ a.e.\ $t\in T$, and $\langle \Phi(t,u(t)),x \rangle-\varphi(t,u(t))<\langle \Phi(t,v(t)),x \rangle-\varphi(t,v(t))$ on $A$ and $\langle \Phi(t,u(t)),x \rangle-\varphi(t,u(t))=\langle \Phi(t,a),x \rangle-\varphi(t,a)$ on $T\setminus A$. Integrating the inequality over $A$ and the equality over $T\setminus A$ and adding both sides of the integrals, we derive a contradiction to \eqref{eq6}. 

(ii) $\Rightarrow$ (i): Conversely, suppose that $-x\in N_{C}(\mathit{w}^*\text{-}\int\Phi(t,u(t))d\mu)$ and $u$ satisfies \eqref{mp*}. Choose any $v\in \U$. Integrating both sides of the inequality $\varphi(t,v(t))-\varphi(t,u(t))\ge \langle \Phi(t,v(t))-\Phi(t,u(t)),x \rangle$ that stems from \eqref{mp*} yields inequality \eqref{eq6} with $\langle \mathit{w}^*\text{-}\int\Phi(t,v(t))d\mu-\mathit{w}^*\text{-}\int\Phi(t,u(t))d\mu,x \rangle\ge 0$, which follows from the fact that $\mathit{w}^*\text{-}\int\Phi(t,v(t))d\mu\in C$ and the definition of the normal cone $N_{C}(\mathit{w}^*\text{-}\int\Phi(t,u(t))d\mu)$. Hence, $u$ is a solution to \eqref{P^*_0}. To prove that $x\in \partial V(0)$, take any $x^*\in E^*$ such that $\U_{x^*}$ is nonempty. Since $\mathit{w}^*\text{-}\int\Phi(t,v(t))d\mu-x^*\in C$ for every $v\in \U_{x^*}$, we have 
$$
\left\langle \mathit{w}^*\text{-}\int_T\Phi(t,v(t))d\mu-x^*-\mathit{w}^*\text{-}\int_T\Phi(t,u(t))d\mu,x \right\rangle \ge 0.
$$
It follows from \eqref{eq5} that $\int\varphi(t,v(t))d\mu-\int\varphi(t,u(t))d\mu\ge \langle x^*,x \rangle$ for every $v\in \U_{x^*}$. Therefore, we obtain $V(x^*)-V(0)\ge \langle x^*,x \rangle$ for every $x^*\in E^*$ with $\U_{x^*}\ne \emptyset$. 
\end{proof}

The multifunction $\Gamma_{\varphi,\Phi}^U:T\twoheadrightarrow \R\times E^*$ is defined similarly in the previous subsection and the same hypothesis with $(\mathrm{H}_4)$ is imposed in the dual space setting, where $\mathit{w}^*\text{-}\mathrm{int}$ means the interior with respect to the weak$^*\!$ topology of $E^*$. 

\begin{description}
\item[$\mathbf{(H_4^*)}$] Either $\left( \mathit{w}^*\text{-}\mathrm{int}\left( \mathit{w}^*\text{-}\displaystyle\int_T\Gamma_{\varphi,\Phi}^U(t)d\mu \right)\right)\cap (\R\times C)$ or $\displaystyle\mathit{w}^*\text{-}\int_T\Phi(t,U(t))d\mu\cap \mathit{w}^*\text{-}\mathrm{int}\,C$ is nonempty. 
\end{description}

\begin{thm}
\label{thm6}
Let $(T,\Sigma,\mu)$ be a saturated finite measure space and $C$ be a nonempty convex subset of $E^*$. Under hypotheses $(\mathrm{H}_1)$, $(\mathrm{H}_2^*)$, $(\mathrm{H}_3)$, and $(\mathrm{H}_4^*)$, $u\in \U$ is a solution to \eqref{P^*_0} if and only if there exists 
$$
-x\in N_{C}\left( \mathit{w}^*\text{-}\int_T\Phi(t,u(t))d\mu \right)
$$ 
such that \eqref{mp*} holds. 
\end{thm}

\begin{proof}
Suppose that $u\in \U$ is a solution to \eqref{P^*_0}. In view of the proof of Theorem \ref{thm5}, it suffices to show that there exists $-x\in N_{C}(\mathit{w}^*\text{-}\int\Phi(t,u(t))d\mu)$ such that inequality \eqref{eq6} holds for every $v\in \M(T,X)$ with $v(t)\in U(t)$ a.e.\ $t\in T$. Toward this end, define the convex subset $D$ of $\R\times E^*$ by $D:=\{ r\in \R\mid r<\int\varphi(t,u(t))d\mu \}\times C$. Note that either $\mathit{w}^*\text{-}\mathrm{int}\,D$ or $\mathit{w}^*\text{-}\mathrm{int}(\mathit{w}^*\text{-}\int\Gamma_{\varphi,\Phi}^U d\mu)$ is nonempty in view of $(\mathrm{H}_4^*)$ and $\mathit{w}^*\text{-}\int\Gamma_{\varphi,\Phi}^U d\mu$ is convex because of Proposition \ref{lyp}. Since $u$ is a solution to \eqref{P^*_0}, we have $D\cap \mathit{w}^*\text{-}\int\Gamma_{\varphi,\Phi}^U d\mu=\emptyset$. By the separation theorem, there exists $(\alpha,y)\in (\R\times E)\setminus \{ 0 \}$ that separates the weak$^*\!$ closure $\overline{D}^{\,\mathit{w}^*}$ of $D$ and $\mathit{w}^*\text{-}\int\Gamma_{\varphi,\Phi}^U d\mu$, that is, 
\begin{equation}
\label{eq7}
\alpha r+\langle x^*,y \rangle \le \alpha s+\langle y^*,y \rangle
\end{equation}
for every $(r,x^*)\in \overline{D}^{\,\mathit{w}^*}$ and $(s,y^*)\in \mathit{w}^*\text{-}\int\Gamma_{\varphi,\Phi}^U d\mu$. Since 
$$
\left( \mathit{w}^*\text{-}\int_T(\varphi(t,u(t)),\Phi(t,u(t)))d\mu \right)\in \overline{D}^{\,\mathit{w}^*}\cap \mathit{w}^*\text{-}\int_T\Gamma_{\varphi,\Phi}^U(t) d\mu,
$$
we have $\langle x^*,y \rangle\le \langle \int\Phi(t,u(t))d\mu,y \rangle$ for every $x^*\in C$. This means that $y\in N_C(\mathit{w}^*\text{-}\int\Phi(t,u(t))d\mu)$. It follows from \eqref{eq7} that $\alpha\ge 0$. Assume that $\alpha=0$. We then have $\langle x^*,y \rangle\le \langle y^*,y \rangle$ for every $x^*\in C$ and $y^*\in \mathit{w}^*\text{-}\int\Phi(t,U(t))d\mu$. If $(\mathit{w}^*\text{-}\mathrm{int}(\mathit{w}^*\text{-}\int\Gamma_{\varphi,\Phi}^Ud\mu))\cap (\R\times C)$ is nonempty, then there exists $y^*_0\in (\mathrm{int}(\mathit{w}^*\text{-}\int\Phi(t,U(t))d\mu))\cap C$ such that for every $h^*\in E^*$ we have $y^*_0\pm \varepsilon h^*\in \mathrm{int}(\mathit{w}^*\text{-}\int\Phi(t,U(t))d\mu$ for all sufficiently small $\varepsilon>0$, and hence, $\langle x^*,y \rangle\le \langle y^*_0\pm\varepsilon h^*,y \rangle$ for every $x^*\in C$. Letting $x^*=y^*_0$ inevitably yields $y=0$, a contradiction. If $\mathit{w}^*\text{-}\int\Phi(t,U(t))d\mu\cap \mathit{w}^*\text{-}\mathrm{int}\,C$ is nonempty, then there exists $x^*_0\in \mathit{w}^*\text{-}\int\Phi(t,U(t))d\mu\cap \mathrm{int}\,C$ such that for every $h^*\in E^*$ we have $x^*_0\pm \varepsilon h^*\in \mathrm{int}\,C$ for all sufficiently small $\varepsilon>0$, and hence, $\langle x^*_0\pm\varepsilon h^*,y \rangle\le \langle y^*,y \rangle$ for every $y^*\in \mathit{w}^*\text{-}\int\Phi(t,U(t))d\mu$. Letting $y^*=x^*_0$ inevitably yields $y=0$, a contradiction again. Henceforth, $\alpha>0$. The normalization $x=-y/\alpha$ in \eqref{eq7} implies inequality \eqref{eq6} with $-x\in  N_{C}(\mathit{w}^*\text{-}\int\Phi(t,u(t))d\mu)$. 

The converse implication is evident from Theorem \ref{thm5}. 
\end{proof}

\begin{description}
\item[$(\mathbf{H_4'^*})$] $\displaystyle \mathit{w}^*\text{-}\int_T\Phi(t,U(t))d\mu\cap C$ is nonempty.  

\item[$(\mathbf{H_5^*})$] $\Gamma_{\varphi,\Phi}^U:T\twoheadrightarrow \R\times E^*$ is integrably bounded with weakly$^*\!$ compact values.  
\end{description}

\begin{thm}
\label{thm7}
Let $(T,\Sigma,\mu)$ be a nonatomic finite measure space and $C$ be a nonempty weakly$\!^*$ closed subset of $E^*$. If $(T,\Sigma,\mu)$ is saturated, then for every quartet $(\varphi,\Phi,U,C)$ satisfying hypotheses $(\mathrm{H}_1)$, $(\mathrm{H}_2^*)$, $(\mathrm{H}_3)$, $(\mathrm{H}_4'^*)$, and $(\mathrm{H}_5^*)$, a solution to \eqref{P^*_0} exists. Conversely, if a solution to \eqref{P^*_0} exists for every quartet $(\varphi,\Phi,U,C)$ satisfying $(\mathrm{H}_1)$, $(\mathrm{H}_2^*)$, $(\mathrm{H}_3)$, $(\mathrm{H}_4'^*)$, and $(\mathrm{H}_5^*)$, then $(T,\Sigma,\mu)$ is saturated whenever $E$ is infinite dimensional and $X$ is an uncountable compact Polish space.  
\end{thm}

\begin{proof}
Let $\{ u_n \}_{n\in \N}\subset \M(T,X)$ be a minimizing sequence of \eqref{P^*_0} satisfying $\int \varphi(t,u_n(t))d\mu\to \inf_{u\in \U}\int\varphi(t,u(t))d\mu$ with $\mathit{w}^*\text{-}\int\Phi(t,u_n(t))d\mu\in C$ and $u_n(t)\in U(t)$ a.e.\ $t\in T$ for each $n\in \N$. Since it follows from Proposition \ref{lyp} that the Gelfand integral $\mathit{w}^*\text{-}\int\Gamma_{\varphi,\Phi}^Ud\mu$ is a weakly$^*\!$ compact set containing the sequence $\{ \mathit{w}^*\text{-}\int(\varphi(t,u_n(t)),\Phi(t,u_n(t)))d\mu \}_{n\in \N}$ and weakly$^*\!$ compact subsets of the dual space of a separable Banach space are metrizable with respect to the weak$^*\!$ topology (see \cite[Theorem 6.30]{ab06}), the weakly$^*\!$ compact set $\mathit{w}^*\text{-}\int\Gamma_{\varphi,\Phi}^Ud\mu$ is also sequentially weakly$^*\!$ compact. Hence, there exists a subsequence of it (which we do not relabel) such that 
$$
(r,x^*):=\mathit{w}^*\text{-}\lim_{n\to \infty}\mathit{w}^*\text{-}\int_T(\varphi(t,u_n(t)),\Phi(t,u_n(t)))d\mu\in \mathit{w}^*\text{-}\int_T\Gamma_{\varphi,\Phi}^U(t)d\mu.
$$
By virtue of the Filippov implicit function theorem (see \cite[Theorem III.38]{cv77}), there exists $\hat{u}\in \M(T,X)$ with $\hat{u}(t)\in U(t)$ a.e.\ $t\in T$ such that $(r,x^*)=\mathit{w}^*\text{-}\int(\varphi(t,\hat{u}(t)),\Phi(t,\hat{u}(t)))d\mu$. We thus obtain 
$$
\int_T\varphi(t,\hat{u}(t))d\mu=\lim_{n\to\infty}\int_T\varphi(t,u_n(t))d\mu=\inf_{u\in \U}\int_T\varphi(t,u(t))d\mu
$$ 
and 
$$
\mathit{w}^*\text{-}\int_T\Phi(t,\hat{u}(t))d\mu=\mathit{w}^*\text{-}\lim_{n\to \infty}\int_T\Phi(t,u_n(t))d\mu\in C.
$$
Therefore, $\hat{u}$ is a solution to \eqref{P^*_0}. 

The converse implication follows from \cite[Theorem 4.4]{sa17}. 
\end{proof}

\begin{thm}
Let $(T,\Sigma,\mu)$ be a saturated finite measure space and $C$ be a nonempty, weakly$^*\!$ closed, convex subset of $E^*$. Then under hypotheses $(\mathrm{H}_1)$, $(\mathrm{H}_2^*)$, $(\mathrm{H}_3)$, $(\mathrm{H}_4'^*)$, and $(\mathrm{H}_5^*)$, $V$ is a proper convex function that is continuous on $\mathrm{int}(\mathrm{dom}\,V)$. Moreover, $V$ is weakly$^*\!$ lower semicontinuous on $E^*$. 
\end{thm}

\begin{proof}
Since the Gelfand integral $\mathit{w}^*\text{-}\int\Gamma_{\varphi,\Phi}^U d\mu$ is weakly$^*\!$ compact and convex by Proposition \ref{lyp}, the convexity of $V$ follows from the equality
$$
V(x^*)=\min\left\{ r\in \R\mid (r,y^*)\in \mathit{w}^*\text{-}\int_T\Gamma_{\varphi,\Phi}^U(t)d\mu\cap(\R\times (C+x^*)) \right\}. 
$$
The properness of $V$ follows from the fact that $\Gamma_{\varphi,\Phi}^U$ is integrably bounded and $0\in \mathrm{dom}\,V$. Since $V$ is bounded from above on $\mathrm{dom}\,V$, it is continuous on $\mathrm{int}(\mathrm{dom}\,V)$.

To demonstrate the weak$^*\!$ lower semicontinuity of $V$, it suffices to show that $\mathrm{epi}\,V=\{ (r,x^*)\in \R\times E^*\mid V(x^*)\le r \}$ is weakly$^*\!$ closed. Toward this end, let $\{ (r_\alpha,x_\alpha^*) \}_{\alpha\in \Lambda}$ be a net in $\mathrm{epi}\,V$ converging weakly$^*\!$ to $(r,x^*)\in \R\times E^*$. Assume, by way of contradiction, that $(r,x^*)$ does not belong to $\mathrm{epi}\,V$. We then have $\lim_\alpha r_\alpha=r<V(x^*)$. Since $V(x^*_\alpha)\le r_\alpha$, extracting a subnet from $\{ x^*_\alpha \}_{\alpha\in \Lambda}$ (which we do not relabel) yields $\U_{x^*_\alpha}\ne \emptyset$ for each $\alpha\in \Lambda$. Replacing $C$ by $C+x^*_\alpha$ in Theorem \ref{thm7} implies that there exists a net $\{ u_\alpha \}_{\alpha\in \Lambda}$ in $\M(T,X)$ such that $V(x^*_\alpha)=\int\varphi(t,u_\alpha(t))d\mu$, $\mathit{w}^*\text{-}\int\Phi(t,u_\alpha(t))d\mu\in C+x^*_\alpha$, and $u_\alpha(t)\in U(t)$ a.e.\ $t\in T$ for each $\alpha\in \Lambda$. Since the weakly$^*\!$ compact set $\mathit{w}^*\text{-}\int\Gamma_{\varphi,\Phi}^Ud\mu$ contains the net $\{ \mathit{w}^*\text{-}\int(\varphi(t,u_\alpha(t)),\Phi(t,u_\alpha(t)))d\mu \}_{\alpha\in \Lambda}$, there exists a subnet of it (which we do not relabel) such that 
$$
(s,y^*):=\mathit{w}^*\text{-}\lim_\alpha\mathit{w}^*\text{-}\int_T(\varphi(t,u_\alpha(t)),\Phi(t,u_\alpha(t)))d\mu\in \mathit{w}^*\text{-}\int_T\Gamma_{\varphi,\Phi}^U(t)d\mu.
$$ 
By the Filippov implicit function theorem, there is $u\in \M(T,X)$ with $u(t)\in U(t)$ a.e.\ $t\in T$ such that $(s,y^*)=\mathit{w}^*\text{-}\int(\varphi(t,u(t)),\Phi(t,u(t)))d\mu$. Since $\lim_{\alpha}\int\varphi(t,u_\alpha(t))d\mu=\int\varphi(t,u(t))d\mu$ and 
$$
\mathit{w}^*\text{-}\lim_{\alpha}\mathit{w}^*\text{-}\int_T\Phi(t,u_\alpha(t))d\mu=\mathit{w}^*\text{-}\int_T\Phi(t,u(t))d\mu\in C+x^*,
$$
we obtain $r<V(x^*)\le \int\varphi(t,u(t))d\mu=\lim_{\alpha}V(x^*_\alpha)\le \lim_{\alpha}r_\alpha=r$, a contradiction. This means that $\mathrm{epi}\,V$ is weakly$^*$ closed. 
\end{proof}


\begin{thebibliography}{00}
\bibitem[Aliprantis and Border(2006)]{ab06}
C.\,D. Aliprantis, K.\,C. Border: Infinite Dimensional Analysis: A Hitchhiker's Guide, 3rd edition, Springer, Berlin (2006). 

\bibitem[Arkin and Levin(1972)]{al72}
V.\,I. Arkin, V.\,L. Levin: Convexity of vector integrals, theorems on measurable choice and variational problems, Russian Math.\ Surveys 2 (1972) 1--85.

\bibitem[Artstein(1974)]{ar74}
Z. Artstein: On a variational problem, J.\ Math.\ Anal.\ Appl.\ 45 (1974) 404--415.

\bibitem[Artstein(1980)]{ar80}
Z. Artstein: Generalized solutions to continuous-\hspace{0pt}time allocation processes, Econometrica 48 (1980) 899--922.

\bibitem[Artstein(2018a)]{ar19}
Z. Artstein: A variational problem determined by probability measures, Optimization 68 (2019) 81--98.

\bibitem[Aumann and Perles(1965)]{ap65}
R.\,J. Aumann, M. Perles: A variational problem arising in economics, J.\ Math.\ Anal.\ Appl.\ 11 (1965) 488--503.

\bibitem[Berliocchi and Lasry(1973)]{bl73}
H. Berliocchi, J.-M. Lasry: Int\'{e}grandes normales et measures param\'{e}tr\'{e}es en calcul des variations, Bulletin de la S.M.F.\ 101 (1973) 129--184.

\bibitem[Castaing and Valadier(1977)]{cv77}
C. Castaing, M. Valadier: Convex Analysis and Measurable Multifunctions, Lect.\ Notes Math.\ 580, Springer, Berlin (1977).

\bibitem[Crasta(1998)]{cr98}
G. Crasta: Existence of minimizers for nonconvex variational problems with slow growth, J.\ Optim.\ Theory Appl.\ 99 (1998) 381--401.

\bibitem[Diestel and Uhl(1977)]{du77}
J. Diestel, J.\,J. Uhl, Jr.: Vector Measures, Amer.\ Math.\ Soc., Providence (1977).

\bibitem[Fajardo and Keisler(2002)]{fk02}
S. Fajardo, H.\,J. Keisler: Model Theory of Stochastic Processes, A K Peters, Ltd., Natick (2002). 

\bibitem[Flores-Baz\'{a}n et al.(2014)]{fbjm14}
F. Flores-Baz\'{a}n, A. Jourani, G. Mastroeni: On the convexity of the value function for a class of nonconvex variational problems: Existence and optimality conditions, SIAM J.\ Control Optim.\ 52 (2014) 3673--3693. 

\bibitem[Fonseca and Leoni(2007)]{fl07}
I. Fonseca, G. Leoni: Modern Methods in the Calculus of Variations: $L^p$ Spaces, Springer, Berlin (2007).  

\bibitem[Fremlin(2012)]{fr12}
D.\,H. Fremlin: Measure Theory, Vol.\,3: Measure Algebras, Part I, 2nd edition, Torres Fremlin, Colchester (2012). 

\bibitem[Greinecker and Podczeck(2013)]{gp13}
M. Greinecker, K. Podczeck: Liapounoff's vector measure theorem in Banach spaces and applications to general equilibrium theory, Econom.\ Theory Bull.\ 1 (2013) 157--173.

\bibitem[Hildenbrand(1974)]{hi74}
W. Hildenbrand: Core and Equilibria of a Large Economy, Princeton Univ.\ Press, Princeton (1974).  

\bibitem[Hoover and Keisler(1984)]{hk84}
D. Hoover, H.\,J. Keisler: Adapted probability distributions, Trans.\ Amer.\ Math.\ Soc.\ 286 (1984) 159--201.

\bibitem[Ioffe(2006)]{io06}
A.\,D. Ioffe: Variational problem associated with a model of welfare economics with a measure spaces of agents, Adv.\ Math.\ Econ.\ 8 (2006) 297--314.

\bibitem[Ioffe(2009)]{io09}
A.\,D. Ioffe: An invitation to tame optimization, SIAM J.\ Optim.\ 19 (2009) 1894--1917.

\bibitem[Ioffe and Tihomirov(1974)]{it74}
A.\,D. Ioffe, V.\,M. Tihomirov: Theory of Extremal Problems, (in Russian), Nauka, Moscow (1974). English translation, North--\hspace{0pt}Holland, Amsterdam (1979).

\bibitem[Kakutani(1944)]{ka44}
S. Kakutani: Construction of a non-separable extension of the Lebesgue measure space, Proc.\ Imp.\ Acad.\ 20 (1944) 115--119. 

\bibitem[Keisler and Sun(2009)]{ks09}
H.\,J. Keisler, Y.\,N. Sun: Why saturated probability spaces are necessary, Adv.\ Math.\ 221 (2009) 1584--1607.

\bibitem[Khan and Sagara(2013)]{ks13}
M.\,A. Khan, N. Sagara: Maharam-\hspace{0pt}types and Lyapunov's theorem for vector measures on Banach spaces, Illinois J.\ Math.\ 57 (2013) 145--169. 

\bibitem[Khan and Sagara(2014)]{ks14}
M.\,A. Khan, N. Sagara: The bang-\hspace{0pt}bang, purification and convexity principles in infinite dimensions: Additional characterizations of the saturation property, Set-\hspace{0pt}Valued Var.\ Anal.\ 22 (2014) 721--746.

\bibitem[Khan and Sagara(2015)]{ks15}
M.\,A. Khan, N. Sagara: Maharam-\hspace{0pt}types and Lyapunov's theorem for vector measures on locally convex spaces with control measures, J.\ Convex Anal.\ 22 (2015) 647--672.

\bibitem[Khan and Sagara(2016)]{ks16}
M.\,A. Khan, N. Sagara: Maharam-types and Lyapunov's theorem for vector measures on locally convex spaces without control measures, Pure Appl.\ Func.\ Anal.\ 1 (2016) 47--62. 

\bibitem[Maruyama(1979)]{ma79}
T. Maruyama: An extension of the Aumann--Perles' variational problem, Proc.\ Japan Acad.\ Ser.\ A 55 (1979) 348--352.

\bibitem[Maruyama(1986)]{ma86}
T. Maruyama: Continuity theorem for non-linear integral functionals and Aumann--Perles' variational problem, Proc.\ Japan Acad.\ Ser.\ A 62 (1986) 163--165.

\bibitem[Maharam(1942)]{ma42}
D. Maharam: On homogeneous measure algebras, Proc.\ Natl.\ Acad.\ Sci.\ USA 28 (1942) 108--111.

\bibitem[Olech(1990)]{ol90}
C. Olech: The Lyapunov theorem: Its extensions and applications, in: Methods of Nonconvex Analysis, A. Cellina (ed.), Lect.\  Notes Math.\ 1446, Springer, Berlin (1990) 86--103. 

\bibitem[Podczeck(2008)]{po08}
K. Podczeck: On the convexity and compactness of the integral of a Banach space valued correspondence, J.\  Math. Econom.\ 44 (2008) 836--852.

\bibitem[Sagara(2015)]{sa15}
N. Sagara: An indirect method of nonconvex variational problems in Asplund spaces: The case for saturated measure spaces, SIAM J.\ Control Optim.\ 53 (2015) 336--351. 

\bibitem[Sagara(2017)]{sa17}
N. Sagara: Relaxation and purification for nonconvex variational problems in dual Banach spaces: The minimization principle in saturated measure spaces, SIAM J.\ Control Optim.\ 55 (2017) 3154--3170.

\bibitem[Schwartz(1973)]{sc73}
L. Schwartz: Radon Measures on Arbitrary Topological Spaces and Cylindrical Measures, Oxford Univ.\ Press, London (1973). 

\bibitem[Sun and Yannelis(2008)]{sy08}
Y.\,N. Sun, N.\,C. Yannelis: Saturation and the integration of Banach valued correspondences, J.\ Math.\ Econom.\ 44 (2008) 861--865.

\bibitem[Thomas(1975)]{th75}
G.\,E.\,F. Thomas: Integration of functions with values in locally convex Suslin spaces, Trans.\ Amer.\ Math.\ Soc.\ 212 (1975) 61--81.

\bibitem[Urbinati(2019)]{ur19}
N. Urbinati: A convexity result for the range of vector measures with applications to large economies, J.\ Math.\ Anal.\ Appl.\ 470 (2019) 16--35.
\end{thebibliography}
\end{document}